\documentclass[a4paper,12pt]{article}

\usepackage{amsfonts}
\usepackage{amscd,color}
\usepackage{amsmath,amsfonts,amssymb,amscd}
\usepackage{indentfirst,graphicx,epsfig}
\usepackage{graphicx,psfrag}
\input{epsf}
\usepackage{graphicx}
\usepackage{epstopdf}
\usepackage{caption}

\newtheorem{thm}{Theorem}[section]
\newtheorem{df}[thm]{Definition}

\newtheorem{lem}[thm]{Lemma}

\newtheorem{cor}[thm]{Corollary}
\newtheorem{pro}[thm]{Proposition}
\newenvironment {proof} {\noindent{\em Proof.}}{\hspace*{\fill}$\Box$\par\vspace{4mm}}

\title{Some upper bounds for the $3$-proper \\
index of graphs\footnote{Supported by NSFC No.11371205 and 11531011, ``973" program No.2013CB834204, and PCSIRT.}}

\author{{\small Hong Chang, Xueliang Li, Zhongmei Qin}\\
{\small  Center for Combinatorics and LPMC}\\
{\small Nankai University, Tianjin 300071, P.R. China}\\
{\small Email: changh@mail.nankai.edu.cn, lxl@nankai.edu.cn, qinzhongmei90@163.com}\\
}
\date{}
\begin{document}
\maketitle
\begin{abstract}
A tree $T$ in an edge-colored graph is a {\it proper tree} if no two adjacent edges of $T$ receive the same color. Let $G$ be a connected graph of order $n$ and $k$ be a fixed integer with $2\le k\le n$. For a vertex subset $S \subseteq V(G)$ with $\left|S\right| \ge 2$, a tree containing all the vertices of $S$ in $G$ is called an $S$-tree. An edge-coloring of $G$ is called a \emph{$k$-proper coloring} if for every $k$-subset $S$ of $V(G)$, there exists a proper $S$-tree in $G$. For a connected graph $G$, the \emph{$k$-proper index} of $G$, denoted by $px_k(G)$, is the smallest number of colors that are needed in a $k$-proper coloring of $G$. In this paper, we show that for every connected graph $G$ of order $n$ and minimum degree $\delta \geq 3$, $px_{3}(G)\le n\frac{\ln(\delta+1)}{\delta+1}(1+o_{\delta}(1))+2$. We also prove that for every connected graph $G$ with minimum degree at least $3$, $px_{3}(G) \le px_{3}(G[D])+3$ when $D$ is a connected $3$-way dominating set of $G$ and $px_{3}(G) \le px_{3}(G[D])+1$ when $D$ is a connected $3$-dominating set of $G$. In addition, we obtain tight upper bounds of the 3-proper index for two special graph classes: threshold graphs and chain graphs. Finally, we prove that $px_3(G) \le \lfloor\frac{n}{2}\rfloor$ for any 2-connected graphs with at least four vertices.\\[2mm]
\end{abstract}

\section{Introduction}

All graphs in this paper are undirected, finite and simple. We follow \cite{BM} for graph theoretical notation and terminology not described here. Let $G$ be a graph, we use $V(G), E(G), |G|, \Delta(G)$ and $\delta(G)$ to denote the vertex set, edge set, order (number of vertices), maximum degree and minimum degree of $G$, respectively. For $D \subseteq V(G)$, let $\overline{D}=V(G)\backslash D$, and let $G[D]$ denote the subgraph of $G$ induced from $D$. For $v \in V(G)$, let $N(v)$ denote the set of neighbors of $v$ in $G$. For two disjoint subsets $X$ and $Y$ of $V(G)$, $E[X,Y]$ denotes the set of edges of $G$ between $X$ and $Y$. The \emph{join} of two graphs $G$ and $H$, denoted by $G \vee H$, is the graph obtained from a disjoint union of $G$ and $H$ by adding edges joining every vertex of $G$ to every vertex of $H$.

Let $G$ be a nontrivial connected graph with an associated {\it edge-coloring} $c : E(G)\rightarrow \{1, 2, \ldots, t\}$, $t \in \mathbb{N}$, where adjacent edges may have the same color. If adjacent edges of $G$ are assigned different colors by $c$, then $c$ is a {\it proper (edge-)coloring}. For a graph $G$, the minimum number of colors needed in a proper coloring of $G$ is referred to as the {\it edge-chromatic number} of $G$ and denoted by $\chi'(G)$. A path in an edge-colored graph $G$ is said to be a {\it rainbow path} if no two edges on the path have the same color. The graph $G$ is called {\it rainbow connected} if for any two vertices there is a rainbow path of $G$ connecting them. An edge-coloring of a connected graph is a {\it rainbow connecting coloring} if it makes the graph rainbow connected. For a connected graph $G$, the \emph{rainbow connection number} $rc(G)$ of $G$ is the smallest number of colors that are needed in order to make $G$ rainbow connected.
These concepts of rainbow connection of graphs were introduced by Chartrand et al.~\cite{CJMZ} in 2008. The readers who are interested in this topic can see \cite{LSS, LS} for a survey.

In \cite{COZ}, Chartrand et al. generalized the concept of rainbow connection to rainbow index. A tree $T$ in an edge-colored graph is a {\it rainbow tree} if no two edges of $T$ receive the same color. Let $G$ be a connected graph of order $n$ and $k$ be a fixed integer with $2\le k\le n$. For a vertex subset $S \subseteq V(G)$ with  $\left|S\right| \ge 2$, a tree containing all the vertices of $S$ in $G$ is called an $S$-tree. An edge-coloring of $G$ is called a \emph{$k$-rainbow coloring} if for every $k$-subset $S$ of $V(G)$, there exists a rainbow $S$-tree in $G$. For a connected graph $G$, the \emph{$k$-rainbow index} of $G$, denoted by $rx_k(G)$, is the minimum number of colors that are needed in a $k$-rainbow coloring of $G$. We refer to \cite{CLZ, QXY,CLYZ,LSYZ} for more details.

Motivated by rainbow coloring and proper coloring in graphs, Andrews et al.~\cite{ALLZ} and Borozan et al.~\cite{BFGMMMT} introduced the concept of proper-path coloring. Let $G$ be a nontrivial connected graph with an edge-coloring. A path in $G$ is called a \emph{proper path} if no two adjacent edges of the path are colored with the same color. An edge-coloring of a connected graph $G$ is a \emph{proper-path coloring} if every pair of distinct vertices of $G$ are connected by a proper path in $G$. An edge-colored graph $G$ is {\it proper connected} if any two vertices of $G$ are connected by a proper path. For a connected graph $G$, the {\it proper connection number} of $G$, denoted by $pc(G)$, is defined as the smallest number of colors that are needed in order to make $G$ proper connected. For more details, we refer to \cite{GLQ,LLZ,LWY} and a dynamic survey \cite{LC}.

Inspired by the $k$-rainbow index and the proper-path coloring, Chen et al. \cite{CLL} introduced the concept of $k$-proper index of a connected graph $G$. A tree $T$ in an edge-colored graph is a {\it proper tree} if no two adjacent edges of $T$ receive the same color. Let $G$ be a connected graph of order $n$ and $k$ be a fixed integer with $2\le k\le n$. For a vertex subset $S \subseteq V(G)$ with $\left|S\right| \ge 2$, a tree containing all the vertices of $S$ in $G$ is called an $S$-tree. An edge-coloring of $G$ is called a \emph{$k$-proper coloring} if for every $k$-subset $S$ of $V(G)$, there exists a proper $S$-tree in $G$. In this case, $G$ is called \emph{$k$-proper connected}. For a connected graph, the \emph{$k$-proper index} of $G$, denoted by $px_k(G)$, is defined as the minimum number of colors that are needed in a $k$-proper coloring of $G$. Clearly, when $k=2$, $px_{2}(G)$ is exactly the proper connection number $pc(G)$ of $G$. Hence, we will study $px_{k}(G)$ only for $k$ with $3 \le k\le n$ here. Let $G$ be a nontrivial connected graph of order $n$ and size $m$, it is easy to see that $pc(G) \le px_3(G)\le \cdots \le px_n(G) \le m$.

The rest of this paper is organised as follows. In Section 2, we list some basic definitions and fundamental results on the $k$-proper index of graphs. In Section 3, we study the 3-proper index by using connected 3-way dominating sets and 3-dominating sets. We first show that for every connected graph $G$ with minimum degree at least $3$, $px_{3}(G) \le px_{3}(G[D])+3$ when $D$ is a connected $3$-way dominating set of $G$. Then, we can easily get that for every connected graph $G$ on $n$ vertices with minimum degree $\delta\geq 3$, $px_{3}(G)\le n\frac{\ln(\delta+1)}{\delta+1}(1+o_{\delta}(1))+2$. At last, we show that $px_{3}(G) \le px_{3}(G[D])+1$ when $D$ is a connected $3$-dominating set of $G$. In addition, we obtain the tight upper bounds of the 3-proper index for two special graph classes: threshold graphs and chain graphs. In Section 4, we prove that $px_3(G) \le \lfloor\frac{n}{2}\rfloor$ for any 2-connected graphs with at least four vertices.

\section{Preliminaries}

To begin with this section, we present the following basic concepts.

\begin{df}\label{df1}
Given a graph $G$, a set $D \subseteq V(G)$ is called a dominating set if every vertex of $\overline{D}$ is adjacent to at least one vertex of $D$. Furthermore, if the subgraph $G[D]$ is connected, it is called a connected dominating set of $G$. The domination number $\gamma(G)$ is the number of vertices in a minimum dominating set of $G$. Similarly, the connected dominating number $\gamma_{c}(G)$ is the number of vertices in a minimum connected dominating set of $G$.
\end{df}

\begin{df}\label{df2}
Let $s$ be a positive integer. A dominating set $D$ of $G$ is called an $s$-way dominating set if $d(v)\ge s$ for every vertex $v \in \overline{D}$. In addition, if $G[D]$ is connected, then $D$ is called a connected $s$-way dominating set.
\end{df}

\begin{df}\label{df3}
A set $D \subseteq G$ is called an $s$-dominating set of $G$ if every vertex of $\overline{D}$ is adjacent to at least $s$ distinct vertices of $D$. Furthermore, if $G[D]$ is connected, then $D$ is called a connected $s$-dominating set. Obviously, a (connected) $s$-dominating set is also a (connected) $s$-way dominating set.
\end{df}

\begin{df}\label{df4}
BFS (breadth-first search) is an algorithm for traversing or searching tree or graph data structures. It starts at the tree root and explores the neighbor vertices first, before moving to the next level neighbors. A BFS-tree (breadth-first search tree) is a spanning rooted tree returned by BFS. Let $T$ be a BFS-tree with root $r$. For a vertex $v$, the level of $v$ is the length of the unique $\{v, r\}$-path in $T$, the ancestors of $v$ are the vertices on the unique $\{v, r\}$-path in $T$, the parent of $v$ is its neighbor on the unique $\{v,r\}$-path in $T$. Its other neighbors are called the children of $v$. The siblings of $v$ are the vertices in the same level as $v$. The left (resp. right) siblings of $v$ are the siblings of $v$ visited before (resp. after) $v$ in BFS.
\end{df}

\noindent\textbf{Remark:} BFS-trees have a nice property: every edge of the graph joins vertices on the same level or consecutive levels. It is not possible for an edge to skip a level. Thus, a neighbor of a vertex $v$ has three possibilities: (1) a sibling of $v$; (2) the parent of $v$ or a right sibling of the parent of $v$; (3) a child of $v$ or a left sibling of the children of $v$.

Next, we state some fundamental results on the $k$-proper index of graphs which will be used in the sequel.

\begin{pro}\label{pro2.1}\cite{CLL}
If $G$ is a nontrivial connected graph of order $n \ge 3$ and $H$ is a connected spanning subgraph of $G$, then $px_{k}(G)\le px_{k}(H)$ for each integer $k$ with $3\le k\le n$. In
particular, $px_k(G)\le px_k(T)$ for every spanning tree $T$ of $G$.
\end{pro}

\begin{pro}\label{pro2.3}\cite{CLL}
If $T$ is a nontrivial tree of order $n\ge 3$, then $px_{k}(T)=\chi^{'}(G)=\Delta(G)$ for each integer $k$ with $3\le k\le n$.
\end{pro}

Propositions \ref{pro2.1} and \ref{pro2.3} provide an upper bound of the $k$-proper index for a graph.

\begin{pro}\label{pro2.4}\cite{CLL}
Let $G$ be a nontrivial connected graph of order $n\ge3$. Then, $2 \le px_{3}(G)\le \ldots \le px_{n}(G)\le$ $\min\{\Delta(T): T$ is a spanning tree of $G \}$.
\end{pro}

A \emph{Hamiltonian path} in a graph $G$ is a path containing every vertex of $G$ and a graph having a Hamiltonian path is a \emph{traceable graph}. The following is an immediate consequence of Proposition \ref{pro2.4}.

\begin{cor}\label{cor2.6}\cite{CLL}
If $G$ is a traceable graph of order $n$, then for each integer $k$ with $3\le k\le n$, $px_{k}(G)=2$.
\end{cor}

Obviously, for any integer $k$ with $3 \le k \le n$, $px_{k}(P_{n})=px_{k}(C_{n})=px_{k}(W_{n})=px_{k}(K_{n})=px_{k}(K_{n,n})=2$.

\begin{lem}\label{lem2.7}
If $G$ is a connected graph with order $n_{G}$ and $H$ is a connected subgraph of $G$ with order $n_{H}$, then for each integer $k$ with $3\le k\le n_{H}$, we have $px_{k}(G) \le px_{k}(H)+n_{G}-n_{H}$; for each integer $k$ with $n_{H}\le k\le n_{G}$, we have $px_{k}(G) \le px_{n_{H}}(H)+n_{G}-n_{H}$.
\end{lem}

\begin{proof} Let $G'$ be a graph obtained from $G$ by contracting $H$ to a single vertex. Then, $G'$ is a connected graph of order $n_{G}-n_{H}+1$. Thus, by Proposition \ref{pro2.4}, $px_{k'}(G') \le n_{G}-n_{H}$ for each integer $k'$ with $3\le k' \le n_{G}-n_{H}+1$. Given an edge-coloring of $G'$ with $n_{G}-n_{H}$ colors such that $G'$ is $k^{'}$-proper connected ($3\le k' \le n_{G}-n_{H}+1$). Now, go back to $G$, and color each edge outside $H$ with the color of the corresponding edge in $G'$. For $H$, if $3 \le k \le n_H$, then we assign $px_{k}(H)$ new colors to the edges of $H$ such that $H$ is $k$-proper connected; if $n_H \le k \le n_G$, then we assign $px_{n_H}(H)$ new colors to the edges of $H$ such that $H$ is $n_H$-proper connected.  The resulting edge-coloring makes $G$ $k$-proper connected. Therefore, for each integer $k$ with $3\le k\le n_{H}$, we have $px_{k}(G) \le px_{k}(H)+n_{G}-n_{H}$; for each integer $k$ with $n_{H}\le k\le n_{G}$, we have $px_{k}(G) \le px_{n_{H}}(H)+n_{G}-n_{H}$. This completes the proof.
\end{proof}

\section{The $3$-proper index and connected dominating sets}

In this section, we give some upper bounds of the $3$-proper index for a graph $G$ by using connected $3$-way dominating sets and $3$-dominating sets.

Let $G$ be a graph, $D \subseteq V(G)$, and $v \in \overline{D}$. We call a path $P=v_0v_1\cdots v_{t}$ a \emph{$v-D$ path} if $v_0=v$ and $V(P) \cap D = \{v_t\}$. Two or more paths are called \emph{internally disjoint} if none of them contains an inner vertex of another. If $P$ is edge-colored, then we denote by $end(P)$ the color of the last edge $v_{t-1}v_t$. Now we give our main results.

\begin{thm}\label{thm4.6}
If $D$ is a connected $3$-way dominating set of a connected graph $G$, then $px_{3}(G) \le px_{3}(G[D])+3$. Moreover, this bound is tight.
\end{thm}
\begin{proof}
Let $D$ be a connected 3-way dominating set of a connected graph $G$. For $v \in \overline{D}$, its neighbors in $D$ are called the {\it feet} of $v$, and the corresponding edges are called the {\it legs} of $v$. A set of proper $v-D$ paths $\{P_1, P_2,\ldots, P_{\ell}\}$ are called {\it strong-proper} if $end(P_i)\ne end(P_j) \ (1 \le i < j \le \ell)$. For a vertex $v$ in $\overline{D}$, we call it {\it good} if there are three internally disjoint strong-proper $v-D$ paths. Otherwise, we call $v$ {\it bad}. Denote by $c(e)$ the color of an edge $e$. Let $T$ be a rooted BFS-tree. Pick a vertex $v$ in $T$, and let $\ell(v)$ be the level of $v$, $p(v)$ the parent of $v$, $ch(v)$ the child of $v$, $\alpha(v)$ the ancestor of $v$ in the first level.

We now review the ideas in the proof. At first, we color the edges in $E[D,\overline{D}]$ and $E(G[\overline{D}])$ with three colors from $\{1,2,3\}$ such that every vertex $v$ of $\overline{D}$ is good. Then, we extend the coloring to the whole graph by coloring the edges in $G[D]$ with $px_{3}(G[D])$ fresh colors. Finally, we prove this edge-coloring is a 3-proper coloring of $G$.

Assume that $A_1,\ldots, A_s$, $B_1,\ldots, B_t$, $C_1,\ldots, C_q$ are the connected components of the subgraph $G-D$ such that $\left|A_i\right|=1$ $(1\leq i\leq s)$, $\left|B_j\right|=2$ $(1\leq j\leq t)$ and $\left|C_k\right|\geq 3$ $(1\leq k\leq q)$, where $s$, $t$ and $q$ are nonnegative integers, and $s=0$ or $t=0$ or $q=0$ means that there is no $A_i$-component or $B_j$-component or $C_k$-component.

For each $A_i$ $(1\leq i\leq s)$, let $v$ be an isolated vertex of $A_i$. Then, $v$ has at least three legs, we color one of them with 1, one of them with 2, and all the others with 3. Thus, $v$ is good.

For each $B_j$ $(1\leq j\leq t)$, let $uv$ be the edge of $B_j$. Then, $u$ has at least two legs, we color one of them with 1, and all the others with 2. Also, $v$ has at least two legs. We color one of them with 2, and all the others with 3. Finally, we color $uv$ with 2. Thus, both $u$ and $v$ are good.

For each $C_k$ $(1\leq k\leq q)$, since $\left|C_k\right|\geq 3$, it follows that there exists a vertex $v_0$ in $C_k$ having at least two neighbors in $C_k$. Starting from $v_0$, we construct a BFS-tree $T$ of $C_k$. Suppose that the neighbors of $v_0$ in $C_k$ are $\{v_1, v_2,\ldots, v_p\}$ $(p\geq2)$, which form the first level of $T$. We call the subtree of $T$ rooted at $v_i$ $(1\leq i\leq p-1)$ of {\it type} $\uppercase\expandafter{\romannumeral1}$ and the subtree of $T$ rooted at $v_p$ of {\it type} $\uppercase\expandafter{\romannumeral2}$. There may be many subtrees of type $\uppercase\expandafter{\romannumeral1}$, but only one subtree of type $\uppercase\expandafter{\romannumeral2}$. For each vertex $v$ in $C_k$, we denote one leg of $v$ by $e_v$, the corresponding foot of $v$ by $t(v)$, the unique edge joining $v$ and its parent $p(v)$ in $T$ by $f_v$. Now, we color the edge $e_v$ and $f_v$ as follows: $c(e_{v_0})=3$; $c(f_{v_i})=2$ and $c(e_{v_i})=1$ for $(1\leq i\leq p-1)$; $c(f_{v_p})=1$ and $c(e_{v_p})=2$; for each vertex $v$ in $V(C_k)\setminus\{v_1, v_2,\ldots, v_p\}$, if $\alpha(v)=v_p$, then set $c(f_v)=2$ and $c(e_v)=3$ when $\ell(v)\equiv0 \ (mod \ 3)$, $c(f_v)=1$ and $c(e_v)=2$ when $\ell(v)\equiv 1 \ (mod \ 3)$, $c(f_v)=3$ and $c(e_v)=1$ when $\ell(v)\equiv 2 \ (mod \ 3)$; if $\alpha(v)=v_i$ $(1\leq i\leq p-1)$, then set $c(f_v)=1$ and $c(e_v)=3$ when $\ell(v)\equiv 0 \ (mod \ 3)$, $c(f_v)=2$ and $c(e_v)=1$ when $\ell(v)\equiv 1 \ (mod \ 3)$, $c(f_v)=3$ and $c(e_v)=2$ when $\ell(v)\equiv 2 \ (mod \ 3)$.
Note that the subtrees of the same type are colored in the same way.

Now, for any non-leaf vertex $v$ in $T$, there exist three internally disjoint strong-proper $v-D$ paths. As for  the root $v_0$, $P^{v_0}_{1}=v_0t(v_0)$; $P^{v_0}_{2}=v_0v_1t(v_1)$; $P^{v_0}_{3}=v_0v_pt(v_p)$. As for any other non-leaf vertex $v$ in $T$, $P^{v}_{1}=vt(v)$; $P^{v}_{2}=vp(v)t(p(v))$; $P^{v}_{3}=vch(v)t(ch(v))$. Hence, all the non-leaf vertices of $T$ are good.

It remains to deal with the leaves of $T$. Pick a leaf $w$ in $T$. Since $w$ has no children, it has exactly two internally disjoint strong-proper $w-D$ paths: $P^{w}_{1}=wt(w)$; $P^{w}_{2}=wp(w)t(p(w))$. In order to make $w$ good, we need to provide the third path $P^{w}_{3}$ which is internally disjoint with $P_1^w$ and $P_2^w$. Since $w\in \overline{D}$, we have $d(w)\geq3$. It follows that $w$ has another neighbor which is not $t(w)$, $p(w)$. Let $W=\{w=w_1, w_2,\ldots, w_a\}$ be the children of $p(w)$ such that $w_i \ (1 \le i \le a)$ is a leaf of $T$ and in the subtrees of the same type. Then, $G[W]$ contains a spanning subgraph $H$ which consists of the components of the following two types: (1) a star, (2) an isolated vertex, where the isolated vertices of $H$ are just the isolated vertices of $G[W]$. For each component of type (1), let $S$ be the star and $V(S)=\{w_{i_1},w_{i_2}, \ldots, w_{i_r}\} \ (r \ge 2)$, where $w_{i_1}$ is the central vertex of $S$. Now we recolor the edge $e_{w_{i_1}}$ and color all edges of $S$. If $w_{i_1}$ is in the subtree of type $\uppercase\expandafter{\romannumeral1}$, then recolor $e_{w_{i_1}}$ with 1 and color all edges of $S$ with 2 when $\ell(w_{i_1})\equiv0 \ (mod \ 3)$; recolor $e_{w_{i_1}}$ with 2 and color all edges of $S$ with 3 when $\ell(w_{i_1})\equiv1 \ (mod \ 3)$; recolor $e_{w_{i_1}}$ with 3 and color all edges of $S$ with 1 when $\ell(w_{i_1})\equiv2 \ (mod \ 3)$. If $w_{i_1}$ is in the subtree of type $\uppercase\expandafter{\romannumeral2}$, then recolor $e_{w_{i_1}}$ with 2 and color all edges of $S$ with 1 when $\ell(w_{i_1})\equiv0 \ (mod \ 3)$; recolor $e_{w_{i_1}}$ with 1 and color all edges of $S$ with 3 when $\ell(w_{i_1})\equiv1 \ (mod \ 3)$; recolor $e_{w_{i_1}}$ with 3 and color all edges of $S$ with 2 when $\ell(w_{i_1})\equiv2 \ (mod \ 3)$. Note that the recoloring of the edge $e_{w_{i_1}}$ has no influence on $p(w)$ since $p(w)$ has at least two children in this case. It is easy to check that for the center $w_{i_1}$ of $S$, there exists a required path $P^{w_{i_1}}_{3}=w_{i_1}w_{i_2}t(w_{i_2})$, and for every vertex $w_{i_t}\in S \ (2 \le t \le r)$, there exists a required path $P^{w_{i_t}}_{3}=w_{i_t}w_{i_1}t(w_{i_1})$. Thus, every leaf in the components of type (1) is good.

For each component of type (2), let $w$ be the isolated vertex and $w'$ be another neighbor of $w$. Note that $w' \notin W$. If $w' \in D$, then we color the edge $ww'$ as follows: if $w$ is in the subtree of type $\uppercase\expandafter{\romannumeral1}$, then we color $c(ww')=1$ when $\ell(w)\equiv 0 \ (mod \ 3)$, $c(ww')=2$ when $\ell(w)\equiv 1 \ (mod \ 3)$, $c(ww')=3$ when $\ell(w)\equiv 2 \ (mod \ 3)$; if $w$ is in the subtree of type $\uppercase\expandafter{\romannumeral2}$, then we color $c(ww')=2$ when $\ell(w)\equiv 0 \ (mod \ 3)$, $c(ww')=1$ when $\ell(w)\equiv 1 \ (mod \ 3)$, $c(ww')=3$ when $\ell(w)\equiv 2 \ (mod \ 3)$.  Note that for any vertex $w$ in the component of type (2) satisfying $w' \in D$, we have $P^{w}_{3}=ww'$. Thus, $w$ is good.

Now we suppose $w' \in T$. Then, $w'$ is either a non-leaf vertex of $T$ or a leaf vertex of $T$ with $p(w')\neq p(w)$. Notice that if $e_{w'}$ is recolored, then $w'$ is a good leaf, and $w'$ has a neighbor $w''$ such that $w''$ is a sibling of $w'$. We distinguish the following four cases:

\noindent\textbf{Case 1:} $w$ and $w'$ are in the subtree of type $\uppercase\expandafter{\romannumeral1}$.

Since $T$ is a BFS-tree, we have that $\ell(w')=\ell(w)-1$ or $\ell(w')=\ell(w)$ or $\ell(w')=\ell(w)+1$. Then, we consider the following three subcases.

\textbf{Subcase 1.1:} $\ell(w)\equiv 0 \ (mod \ 3)$.

If $\ell(w')=\ell(w)-1$, then color $ww'$ with 1. Thus, $P^{w}_{3}=ww'p(w')t(p(w'))$. If $w'$ is bad, then $P^{w'}_{3}=w'wt(w)$.

If $\ell(w')=\ell(w)$, then color $ww'$ with 3. Thus, $P^{w}_{3}=ww'p(w')p(p(w'))t(p(p(w')))$. If $w'$ is bad, then $P^{w'}_{3}=w'wp(w)p(p(w))t(p(p(w)))$.

If $\ell(w')=\ell(w)+1$, then color $ww'$ with 2. If $e_{w'}$ is recolored, then $w'$ is already good. Thus, $P^{w}_{3}=ww'w''t(w'')$ (where $w''$ is a sibling of $w'$). If $e_{w'}$ is not recolored, then $P^{w}_{3}=ww't(w')$. In this situation, if $w'$ is bad, then $P^{w'}_{3}=w'wp(w)t(p(w))$.

\textbf{Subcase 1.2:} $\ell(w)\equiv1 \ (mod \ 3)$.

If $\ell(w')=\ell(w)-1$, then color $ww'$ with 2. Thus, $P^{w}_{3}=ww'p(w')t(p(w'))$. If $w'$ is bad, $P^{w'}_{3}=w'wt(w)$.

If $\ell(w')=\ell(w)$, then color $ww'$ with 1. If $w$ and $w'$ are in the first level, then $w'$ has at least one child since $p(w')=p(w)$ and is already good. Thus, $P^{w}_{3}=ww'ch(w')t(ch(w'))$. Now suppose that $w$ and $w'$ are not in the first level. Then, $P^{w}_{3}=ww'p(w')p(p(w'))t(p(p(w')))$. If $w'$ is bad, then $P^{w'}_{3}=w'wp(w)p(p(w))t(p(p(w)))$.

If $\ell(w')=\ell(w)+1$, then color $ww'$ with 3. If $e_{w'}$ is recolored, then $w'$ is already good. Thus, $P^{w}_{3}=ww'w''t(w'')$ (where $w''$ is a sibling of $w'$). If $e_{w'}$ is not recolored, then $P^{w}_{3}=ww't(w')$. In this case, if $w'$ is bad, then $P^{w'}_{3}=w'wp(w)t(p(w))$.

\textbf{Subcase 1.3:} $\ell(w)\equiv 2 \ (mod \ 3)$.

If $\ell(w')=\ell(w)-1$, then color $ww'$ with 3. Thus, $P^{w}_{3}=ww'p(w')t(p(w'))$. If $w'$ is bad, then $P^{w'}_{3}=w'wt(w)$.

If $\ell(w')=\ell(w)$, then color $ww'$ with 2. Thus, $P^{w}_{3}=ww'p(w')p(p(w'))t(p(p(w')))$. If $w'$ is bad, then $P^{w'}_{3}=w'wp(w)p(p(w))t(p(p(w)))$.

If $\ell(w')=\ell(w)+1$, then color $ww'$ with 1. If $e_{w'}$ is recolored, then $w'$ is already good. Thus, $P^{w}_{3}=ww'w''t(w'')$ (where $w''$ is a sibling of $w'$). If $e_{w'}$ is not recolored, then $P^{w}_{3}=ww't(w')$. In this case, if $w'$ is bad, then $P^{w'}_{3}=w'wp(w)t(p(w))$.

Thus, both $w$ and $w'$ are good.

\noindent\textbf{Case 2:} $w$ is in the subtrees of type $\uppercase\expandafter{\romannumeral1}$ and $w'$ is in the subtree of type $\uppercase\expandafter{\romannumeral2}$.

Since $T$ is a BFS-tree, it follows that $\ell(w')=\ell(w)-1$ or $\ell(w')=\ell(w)$. Then, we consider the following three subcases.

\textbf{Subcase 2.1:} $\ell(w)\equiv 0 \ (mod \ 3)$.

If $\ell(w')=\ell(w)-1$, then we distinguish two situations. If $e_{w'}$ is not recolored, then color $ww'$ with 2. Thus, $P^{w}_{3}=ww't(w')$. In this situation, if $w'$ is bad, then $P^{w'}_{3}=w'wt(w)$. If $e_{w'}$ is recolored, then color $ww'$ with 3 and $w'$ is already good. Thus, $P^{w}_{3}=ww'w''t(w'')$ (where $w''$ is a sibling of $w'$).

If $\ell(w')=\ell(w)$, then color $ww'$ with 3. Thus, $P^{w}_{3}=ww'p(w')t(p(w'))$. If $w'$ is bad, then $P^{w'}_{3}=w'wp(w)t(p(w))$.

\textbf{Subcase 2.2:} $\ell(w)\equiv 1 \ (mod \ 3)$.

If $\ell(w')=\ell(w)-1$, then color $ww'$ with 3. Thus, $P^{w}_{3}=ww'p(w')p(p(w'))t(p(p(w')))$. If $w'$ is bad, then $P^{w'}_{3}=w'wp(w)p(p(w))t(p(p(w)))$.

If $\ell(w')=\ell(w)$, then we distinguish two situations. If $e_{w'}$ is not recolored, then color $ww'$ with 3. Thus, $P^{w}_{3}=ww't(w')$. In this situation, if $w'$ is bad, then $P^{w'}_{3}=w'wt(w)$. If $e_{w'}$ is recolored, then color $ww'$ with 2 and $w'$ is already good. Thus, $P^{w}_{3}=ww'w''t(w'')$ (where $w''$ is a sibling of $w'$).

\textbf{Subcase 2.3:} $\ell(w)\equiv 2 \ (mod \ 3)$.

If $\ell(w')=\ell(w)-1$, then color $ww'$ with 2, Thus, $P^{w}_{3}=ww'p(w')t(p(w'))$. If $w'$ is bad, then $P^{w'}_{3}=w'wp(w)t(p(w))$.

If $\ell(w')=\ell(w)$, then color $ww'$ with 1. Thus, $P^{w}_{3}=ww'p(w')p(p(w'))t(p(p(w')))$. If $w'$ is bad, then $P^{w'}_{3}=w'wp(w)p(p(w))t(p(p(w)))$.

Thus, both $w$ and $w'$ are good.

\noindent\textbf{Case 3:} If $w$ is in the subtrees of type $\uppercase\expandafter{\romannumeral2}$ and $w'$ is the subtree of type $\uppercase\expandafter{\romannumeral1}$.

Since $T$ is a BFS-tree, we have that $\ell(w')=\ell(w)$ or $\ell(w')=\ell(w)+1$. Then, we consider the following three subcases.

\textbf{Subcase 3.1:} $\ell(w)\equiv 0 \ (mod \ 3)$.

If $\ell(w')=\ell(w)$, then color $ww'$ with 3. Thus, $P^{w}_{3}=ww'p(w')t(p(w'))$. If $w'$ is bad, then $P^{w'}_{3}=w'wp(w)t(p(w))$.

If $\ell(w')=\ell(w)+1$, then color $ww'$ with 3. Thus, $P^{w}_{3}=ww'p(w')p(p(w'))t(p(p(w')))$. If $w'$ is bad, then $P^{w'}_{3}=w'wp(w)p(p(w))t(p(p(w)))$.

\textbf{Subcase 3.2:} $\ell(w)\equiv 1 \ (mod \ 3)$.

If $\ell(w')=\ell(w)$, then we distinguish two situations. If $e_{w'}$ is not recolored, then color $ww'$ with 3. Thus, $P^{w}_{3}=ww't(w')$. In this situation, if $w'$ is bad, then $P^{w'}_{3}=w'wt(w)$. If $e_{w'}$ is recolored, then color $ww'$ with 2 and $w'$ is already good. Thus, $P^{w}_{3}=ww'w''t(w'')$ (where $w''$ is a sibling of $w'$).

If $\ell(w')=\ell(w)+1$, then color $ww'$ with 2, Thus, $P^{w}_{3}=ww'p(w')t(p(w'))$. If $w'$ is bad, then $P^{w'}_{3}=w'wp(w)t(p(w))$.

\textbf{Subcase 3.3:} $\ell(w)\equiv 2 \ (mod \ 3)$.

If $\ell(w')=\ell(w)$, then color $ww'$ with 1. Thus, $P^{w}_{3}=ww'p(w')p(p(w'))t(p(p(w')))$. If $w'$ is bad, then $P^{w'}_{3}=w'wp(w)p(p(w))t(p(p(w)))$.

If $\ell(w')=\ell(w)+1$, then we distinguish two situations. If $e_{w'}$ is not recolored, then color $ww'$ with 2. Thus, $P^{w}_{3}=ww't(w')$. In this situation, if $w'$ is bad, then $P^{w'}_{3}=w'wt(w)$. If $e_{w'}$ is recolored, then color $ww'$ with 3 and $w'$ is already good. Thus, $P^{w}_{3}=ww'w''t(w'')$ (where $w''$ is a sibling of $w'$).

Thus, both $w$ and $w'$ are good.

\noindent\textbf{Case 4:} If $w$, $w'$ are in the subtree of type $\uppercase\expandafter{\romannumeral2}$.

Since $T$ is a BFS-tree, it follows that $\ell(w')=\ell(w)-1$ or $\ell(w')=\ell(w)$ or $\ell(w')=\ell(w)+1$. Then, we consider the following three subcases.

\textbf{Subcase 4.1:} $\ell(w)\equiv 0 \ (mod \ 3)$.

If $\ell(w')=\ell(w)-1$, then color $ww'$ with 2. Thus, $P^{w}_{3}=ww'p(w')t(p(w'))$. If $w'$ is bad, then $P^{w'}_{3}=w'wt(w)$.

If $\ell(w')=\ell(w)$, then color $ww'$ with 3. Thus, $P^{w}_{3}=ww'p(w')p(p(w'))t(p(p(w')))$. If $w'$ is bad, then $P^{w'}_{3}=w'wp(w)p(p(w))t(p(p(w)))$.

If $\ell(w')=\ell(w)+1$, then color $ww'$ with 1. If $e_{w'}$ is recolored, then $w'$ is already good. Thus, $P^{w}_{3}=ww'w''t(w'')$ (where $w''$ is a sibling of $w'$). If $e_{w'}$ is not recolored, then $P^{w}_{3}=ww't(w')$. In this case, if $w'$ is bad, then $P^{w'}_{3}=w'wp(w)t(p(w))$.

\textbf{Subcase 4.2:} $\ell(w)\equiv 1 \ (mod \ 3)$.

If $\ell(w')=\ell(w)-1$, then color $ww'$ with 1. Thus, $P^{w}_{3}=ww'p(w')t(p(w'))$. If $w'$ is bad, then $P^{w'}_{3}=w'wt(w)$.

If $\ell(w')=\ell(w)$, then color $ww'$ with 2. If $w$ and $w'$ are in the first level, then $w'$ has at least one child since $p(w')=p(w)$ and is already good. Thus, $P^{w}_{3}=ww'ch(w')t(ch(w'))$. Now suppose that $w$ and $w'$ are not in the first level. Then, $P^{w}_{3}=ww'p(w')p(p(w'))t(p(p(w')))$. If $w'$ is bad, then $P^{w'}_{3}=w'wp(w)p(p(w))t(p(p(w)))$.

If $\ell(w')=\ell(w)+1$, then color $ww'$ with 3. If $e_{w'}$ is recolored, then $w'$ is already good. Thus, $P^{w}_{3}=ww'w''t(w'')$ (where $w''$ is a sibling of $w'$). If $e_{w'}$ is not recolored, then $P^{w}_{3}=ww't(w')$. In this case, if $w'$ is bad, then $P^{w'}_{3}=w'wp(w)t(p(w))$.

\textbf{Subcase 4.3:} $\ell(w)\equiv 2 \ (mod \ 3)$.

If $\ell(w')=\ell(w)-1$, then color $ww'$ with 3. Thus $P^{w}_{3}=ww'p(w')t(p(w'))$. If $w'$ is bad, then $P^{w'}_{3}=w'wt(w)$.

If $\ell(w')=\ell(w)$, then color $ww'$ with 1. Thus, $P^{w}_{3}=ww'p(w')p(p(w'))t(p(p(w')))$. If $w'$ is bad, then $P^{w'}_{3}=w'wp(w)p(p(w))t(p(p(w)))$.

If $\ell(w')=\ell(w)+1$, then color $ww'$ with 2. If $e_{w'}$ is recolored, then $w'$ is already good. Thus, $P^{w}_{3}=ww'w''t(w'')$ (where $w''$ is a sibling of $w'$). If $e_{w'}$ is not recolored, then $P^{w}_{3}=ww't(w')$. In this case, if $w'$ is bad, then $P^{w'}_{3}=w'wp(w)t(p(w))$.

Thus, both $w$ and $w'$ are good.

After the above process, $w$ becomes good, and so does $w'$ if $w'$ is bad. Note that all the good vertices are still good since we just color the edge $ww'$. As a result, every vertex in $T$ is good.

If there still remains uncolored edges in $E[D,\overline{D}]$ and $E(G[\overline{D}])$, then color them with 1. Now we have a coloring of all the edges in $E[D,\overline{D}]$ and $E(G[\overline{D}])$ using three colors from $\{1,2,3\}$ such that all the vertices in $\overline{D}$ are good. Next, we color the edges in $G[D]$ with $px_{3}(G[D])$ fresh colors such that for each triple of vertices in $D$, there is a proper tree in $G[D]$ connecting them. Thus, we provide an edge-coloring $c$ of $G$ using $px_{3}(G[D])+3$ colors.

Now we show that this edge-coloring $c$ is a 3-proper coloring of $G$, which implies $px_{3}(G) \le px_{3}(G[D])+3$. We first claim that for any three vertices $u, v, w$ in $\overline{D}$, there exist a proper $u-D$ path $P^u$, a proper $v-D$ path $P^v$ and a proper $w-D$ path $P^w$ such that $P^u \cup P^v \cup P^w$ is also proper. Since this edge-coloring makes every vertex of $\overline{D}$ good, we only need to consider the situation that $u, v, w$ are in the same component of $G-D$. So, $u, v, w \in C_k \ (1 \le k \le q)$. Note that for any vertex $x \ne v_0 \in C_k$, there are three internally disjoint strong-proper $x-D$ paths $P_1^x, P_2^x, P_3^x$ such that $P_1^x=xt(x)$ and $P_2^x=xp(x)t(p(x))$. For $v_0 \in C_k$, the three internally disjoint strong-proper $v_0-D$ paths are $P_1^{v_0}=v_0t(v_0)$, $P_2^{v_0}=v_0v_1t(v_1)$ and $P_3^{v_0}=v_0v_pt(v_p)$. If $\{c(e_u), c(e_v), c(e_w)\}$ contains three distinct colors, then $P_1^u \cup P_1^v \cup P_1^w$ is also proper. If $\{c(e_u), c(e_v), c(e_w)\}$ contains two distinct colors, without loss of generality, assume $c(e_u) \ne c(e_v)$, then it is easy to check that either $P_1^u \cup P_1^v \cup P_2^w$ or $P_1^u \cup P_1^v \cup P_3^w$ is proper. Now we assume that $c(e_u) = c(e_v) = c(e_w)$. If $u, v, w$ are in the subtrees of the same type, then we distinguish the following situations. If one of $\{e_u, e_v, e_w\}$ is recolored, without loss of generality, assume that $e_u$ is recolored, then $P_2^u \cup P_1^v \cup P_2^w$ is proper. If two of $\{e_u, e_v, e_w\}$ are  recolored, without loss of generality, assume $e_w$ is not recolored, then $P_2^u \cup P_1^v \cup P_2^w$ is proper. If $e_u$, $e_v$ and $e_w$ are simultaneously recolored or not recolored, without loss of generality, assume $v$ is visited before $w$ in $T$, then $P_1^u \cup P_2^v \cup P_3^w$ is proper. Now suppose that $u, v, w$ are in the subtrees of different types. Without loss of generality, assume $u, v$ are in the subtree of the same type, and $w$ is in the subtree of the other type. If $e_u$, $e_v$ and $e_w$ are simultaneously recolored or not recolored, then $P_1^u \cup P_2^v \cup P_2^w$ is proper. If $e_u$ and $e_v$ are recolored, $e_w$ is not recolored, then $P_1^u \cup P_3^v \cup P_2^w$ is proper. If one of $\{e_u, e_v\}$ is recolored, $e_w$ is recolored, without loss of generality, assume $e_u$ is recolored, then $P_2^u \cup P_1^v \cup P_2^w$ is proper. If one of $\{e_u, e_v\}$ is recolored, $e_w$ is not recolored, without loss of generality, assume $e_u$ is recolored, then $P_1^u \cup P_2^v \cup P_2^w$ is proper. If $e_u$ and $e_v$ are not recolored, $e_w$ is recolored, then $P_1^u \cup P_2^v \cup P_3^w$ is proper. Thus, the claim holds.

Next, it is sufficient to show that for any three vertices $u$, $v$, $w$ of $G$, there exists a proper tree connecting them. If all of them are in $D$, then there is already a proper tree connecting them in $G[D]$. If one of them is in $\overline{D}$, without loss of generality, say $u \in \overline{D}$, then any leg of $u$ (colored by 1, 2 or 3) together with the proper tree connecting $v$, $w$, and the corresponding foot of $u$ in $G[D]$ forms a proper $\{u, v, w\}$-tree. If two of them are in $\overline{D}$, without loss of generality, say $u,v\in \overline{D}$, then there exist a proper $u-D$ path $P^u$, a proper $v-D$ path $P^v$ such that $P^u\cup P^v$ is also proper. Assume that the endvertices of $P^u$, $P^v$ in $D$ are $u'$, $v'$, respectively. Then, the proper tree connecting $u'$, $v'$ and $w$ together with the paths $P^u$ and $P^v$ forms a proper $\{u,v,w\}$-tree. If all of them are in $\overline{D}$, then there exist a proper $u-D$ path $P^u$, a proper $v-D$ path $P^v$ and a proper $w-D$ path $P^w$ such that $P^u\cup P^v\cup P^w$ is also proper. Assume that the endvertices of $P^u$, $P^v$ and $P^w$ in $D$ are $u'$, $v'$, $w'$, respectively. Then, the proper tree in $G[D]$ connecting $u'$, $v'$, $w'$ together with the paths $P^u$, $P^v$ and $P^w$ forms a proper $\{u, v, w\}$-tree.

To complete the proof of Theorem \ref{thm4.6}, we show the tightness of the bound with the graph class $\mathcal{G}$. Let $p$ be an integer with $p\ge 3$, $\mathcal{G}=$ \{$G$: $G$ is a graph obtained by taking $p$ complete graph $K_{i_{1}}, K_{i_{2}}, \ldots, K_{i_{p}}$ with just a vertex in common, say $v_{0}$  for $i_{j} \ge 4$ when $1\le j\le p$\}. For any graph $G$ in $\mathcal{G}$, it is obvious that $D=\{v_{0}\}$ is a connected $3$-way dominating set. By Theorem \ref{thm4.6}, we have $px_{3}(G)\le px_{3}(G[D])+3=3$. On the other hand, it is easy to show that $px_{3}(G)=3$. Thus, the bound is tight.
\end{proof}

\begin{cor}\label{cor4.7} Let $G$ be a connected graph with minimum degree $\delta(G)\ge 3$. Then, $px_{3}(G)\le \gamma_{c}(G)+2$.
\end {cor}

\begin{proof} Since $\delta(G)\ge 3$, every connected dominating set of $G$ is a connected $3$-way dominating set. Consider a minimum connected dominating set $D$ with size $\gamma_{c}(G)$. Then, $px_{3}(G[D])\le \left|D\right|-1=\gamma_{c}(G)-1$. We have that $px_{3}(G)\le px_{3}(G[D])+3\le \gamma_{c}(G)+2$ by Theorem \ref{thm4.6}.
\end{proof}

Caro et al. \cite{CWY} showed that for every connected graph $G$ of order $n$ and minimum degree $\delta$, $\gamma_{c}(G)=n\frac{\ln(\delta+1)}{\delta+1}(1+o_{\delta}(1))$. With the help of Corollary \ref{cor4.7}, we obtain the following result.

\begin{cor}\label{cor4.8} Let $G$ be a connected graph with minimum degree $\delta(G)\ge 3$. Then, $px_{3}(G)\le n\frac{\ln(\delta+1)}{\delta+1}(1+o_{\delta}(1))+2$.
\end {cor}

Next, we will give another upper bound for the $3$-proper index of graphs with respect to the connected $3$-dominating set.

\begin{thm}\label{thm4.10}
If $D$ is a connected $3$-dominating set of a connected graph $G$ with minimum degree $\delta(G)\ge 3$, then $px_{3}(G) \le px_{3}(G[D])+1$. Moreover, the bound is tight.
\end{thm}
\begin{proof}
Since $D$ is a connected 3-dominating set, every vertex in $\overline{D}$ has at least three neighbors in $D$. Let $t=px_3(G[D])$. We first color the edges in $G[D]$ with $t$ different colors from $\{2, 3, \ldots, t+1\}$ such that for every triple of vertices in $D$, there exists a proper tree in $G[D]$ connecting them. Then, we color the remaining edges with color 1.

Next, we will show that this edge-coloring makes $G$ 3-proper connected. For any triple $\{u, v, w\}$ of vertices in $G$, if all of them are in $D$, then there is already a proper tree connecting them in $G[D]$. If one of them is in $\overline{D}$, without loss of generality, say $u \in \overline{D}$, then let $u_1$ be the neighbor of $u$ in $D$. Thus, the proper tree connecting $u_1, v, w$ in $G[D]$ together with the edge $uu_1$ forms a proper $\{u,v,w\}$-tree in $G$. If two of them are in $\overline{D}$, without loss of generality, say $u, v \in \overline{D}$, then let $u_1$, $v_1$ be the two distinct neighbors of $u, v$ in $D$, respectively. Thus, the proper tree connecting $u_1, v_1, w$ in $G[D]$ together with two edges $uu_1$, $vv_1$ forms a proper $\{u,v,w\}$-tree in $G$. If all of them are in $\overline{D}$, then let $u_1$, $v_1$, $w_1$ be the three distinct neighbors of $u, v, w$ in $D$, respectively. Thus, the proper tree connecting $u_1, v_1, w_1$ in $G[D]$ together with three edges $uu_1$, $vv_1$, $ww_1$ forms a proper $\{u,v,w\}$-tree in $G$.

The tightness of the bound can be seen from the following corollaries.
\end{proof}

Next, we give some tight upper bounds for the 3-proper index of two special graph classes: threshold graphs and chain graphs, which implies the tightness of the bound in Theorem \ref{thm4.10}. A graph $G$ is called a \emph{threshold graph}, if there exists a weight function $w$: $V(G)\rightarrow \mathbb{R}$ and a real constant $t$ such that two vertices $u$, $v\in V(G)$ are adjacent if and only if $w(u)+w(v)\geq t$. We call $t$ the threshold of $G$. A bipartite graph $G(U,V)$ is called a \emph{chain graph}, if the vertices of $U$ can be ordered as $U=\{u_{1},u_{2},\ldots,u_{s}\}$ such that $N(u_{1})\subseteq N(u_{2})\subseteq\cdots\subseteq N(u_{s})$.

\begin{cor}
Let $G$ be a connected threshold graph with $\delta(G)\ge 3$. Then, $px_{3}(G) \le 3$, and the bound is tight.
\end{cor}
\begin{proof}
Suppose that $V(G)=\{v_{1},v_{2},\ldots,v_{n}\}$ where $w(v_{1})\ge w(v_{2})\ge \cdots\ge w(v_{n})$. Since $\delta(G)\ge 3$, $v_{1}$, $v_{2}$, $v_{3}$ are adjacent to all the other vertices in $G$. Thus, $D=\{v_{1},v_{2},v_{3}\}$ is a connected $3$-dominating set of $G$. Since $G[D]=K_{3}$, we have $px_{3}(G[D])=2$. It follows that $px_{3}(G) \le px_{3}(G[D])+1= 3$ by Theorem \ref{thm4.10}.

Next, we give a class of threshold graphs which have $px_3(G)=3$. Consider the graph $G=rK_1 \vee K_3$, where $r \ge 2 \times 2^3+1$. Let $V(rK_1)=\{v_1, v_2, \ldots, v_r\}$ and $V(K_3)=\{u_1, u_2, u_3\}$. Obviously, it is a threshold graph ($u_1,u_2,u_3$ can be given a weight 1, others a weight 0 and the threshold 1). We will show that $px_3(G)\geq 3$. By contradiction, we assume that $G$ has a 3-proper coloring with 2 colors. For each vertex $v_{i}\in rK_1$, there exists a $3$-tuple $C_{i}=(c_{1}, c_{2}, c_{3})$ so that $c(v_{i}u_{j})=c_{j}$ for $1 \le j \le 3$. Therefore, each vertex $v_{i}\in rK_1$ has $2^{3}$ different ways of coloring its incident edges using $2$ colors. Since $r\ge 2 \times 2^3+1$, there exist at least three vertices $v_{i}$, $v_{j}, v_{k} \in V$ such that $C_{i}=C_{j}=C_k$. It is easy to check that there is no proper tree connecting $v_{i}$, $v_{j}, v_{k}$ in $G$, a contradiction.
\end{proof}

\begin{cor}
Let $G$ be a connected chain graph with $\delta(G)\ge 3$. Then, $px_{3}(G) \le 3$, and the bound is tight.
\end{cor}
\begin{proof}
Let $G=G(U,V)$ be a connected chain graph, where $U=\{u_{1},u_{2},\ldots,u_{s}\}$, $V=\{v_{1},v_{2},\ldots, v_{t}\}$ such that $N(u_{1})\subseteq N(u_{2})\subseteq\cdots\subseteq N(u_{s})$. Since the minimum degree of $G$ is at least three, $u_i (s-2 \le i \le s)$ is adjacent to all the vertices in $V$, and $N(u_1)$ has at least three vertices, say $\{v_1, v_2, v_3\}$. Clearly, $v_1, v_2, v_3$ are adjacent to all the vertices in $V$. Therefore, $D=\{v_{1},v_{2},v_{3},u_{s-2},u_{s-1},u_{s}\}$ is a connected $3$-dominating set of $G$. Moreover, $G[D]=K_{3,3}$ is a traceable graph, we have $px_{3}(K_{3,3})=2$. By Theorem \ref{thm4.10} we have that $px_{3}(G) \le px_{3}(K_{3,3})+1\le 3$.

Now, we give a class of chain graphs which have $px_3(G)=3$. Consider the chain graph $G=G[U,V]$, where $U=\{u_{1},u_{2},\ldots,u_{s}\}$, $V=\{v_{1},v_{2},\ldots v_{t}\}$ such that $N(u_{1})= N(u_{2})=\cdots= N(u_{s-3})=\{v_1,v_2,v_3\}$, $N(u_{s-2})= N(u_{s-1})=N(u_s)=\{v_1,v_2, \ldots,v_t\}$ and $t \ge 2 \times 2^3 +4$. Next, we show that $px_3(G)\geq3$. Suppose not, we assume that $G$ has a 3-proper coloring with 2 colors. For each vertex $v_{i}\in V$ for $4 \le i \le t$, there exists a $3$-tuple $C_{i}=(c_{1}, c_{2}, c_{3})$ such that $c(u_{j}v_{i})=c_{j}$ for $s-2 \le j \le s$. Therefore, each vertex $v_{i}\in V$ ($4 \le i \le t$) has $2^{3}$ different ways of coloring its incident edges using $2$ colors. Since $t-3 \ge 2 \times 2^3+1$, there exist at least three vertices $v_{i}$, $v_{j}, v_{k} \in V\setminus \{v_1, v_2, v_3\}$ such that $C_{i}=C_{j}=C_k$. It is easy to check that there is no proper tree connecting $v_{i}$, $v_{j}, v_{k}$ in $G$, a contradiction.
\end{proof}

\section{The $3$-proper index of 2-connected graphs}

In this section, we give an upper bound for the $3$-proper index of $2$-connected graphs. The following notation and terminology are needed in the sequel.

\begin{df}
Let $F$ be a subgraph of a graph $G$. An {\it ear} of $F$ in $G$ is a nontrivial path in $G$ whose endvertices are in $F$ but whose internal vertices are not. A {\it nested sequence} of graphs is a sequence $(G_0, G_1, \ldots, G_k)$ of graphs such that $G_i\subset G_{i+1}$, $0 \le i <k$. An {\it ear-decomposition} of a 2-connected graph $G$ is a nested sequence $(G_0, G_1, \ldots, G_k)$ of 2-connected subgraphs of $G$ such that: (1) $G_0$ is a cycle; (2) $G_i=G_{i-1} \cup P_i$, where $P_i$ is an ear of $G_{i-1}$ in $G$, $1 \le i \le k$; (3) $G_k=G$.
\end{df}

From Corollary \ref{cor2.6}, we have that if $G$ is a 2-connected Hamiltonian graph of order $n \ (n \ge 3)$, then $px_3(G)=2$. Thus, we only need to consider the non-Hamiltonian graphs.

Let $G$ be a 2-connected non-Hamiltonian graph of order $n \ (n \ge 4)$. Then, $G$ must have an even cycle. In fact, since $G$ is 2-connected, $G$ must have a cycle $C$. If $C$ is an even cycle, we are done. Otherwise, $C$ is an odd cycle, we then choose an ear $P$ of $C$ such that $V(C) \cap V(P)=\{a,b\}$. Since the lengths of the two segments between $a,b$ on $C$ have different parities, $P$ joining one of the two segments forms an even cycle. Then, starting from an even cycle $G_0$, there exists a nonincreasing ear-decomposition $(G_0, G_1, \ldots, G_t, G_{t+1}, \ldots, G_k)$ of $G$, such that $G_i=G_{i-1} \cup P_i \ (1 \le i \le k)$ and $P_i$ is a longest ear of $G_{i-1}$, i.e., $\ell(P_1) \ge \ell(P_2) \ge \cdots \ge \ell(P_k)$, where $\ell(P_i)$ denotes the length of $P_i$. Suppose that $V(P_i) \cap V(G_{i-1})=\{a_i,b_i\} \ (1 \le i \le k)$. We call the distinct vertices $a_i, b_i \ (1 \le i \le k)$ the \emph{endpoints} of the ear $P_i$, the edges incident to the endpoints in $P_i$ the end-edges of $P_i$, the other edges the internal edges of $P_i$. Without loss of generality, suppose that $\ell(P_t) \ge 2$ and $\ell(P_{t+1})= \cdots= \ell(P_k)=1$. So, $G_t$ is a 2-connected spanning subgraph of $G$. Since $G$ is non-Hamiltonian graph, we have $t \ge 2$. Denote the order of $G_i \ (0 \le i \le k)$ by $n_i$.

\begin{thm}
Let $G$ be a 2-connected non-Hamiltonian graph of order $n \ (n \ge 4)$. Then, $px_3(G) \le \lfloor\frac{n}{2}\rfloor$.
\end{thm}
\begin{proof}
Since $G_t \ (t \ge 2)$ in the nonincreasing ear-decomposition is a 2-connected spanning subgraph of $G$, it only needs to show that $G_t$ has a 3-proper coloring with at most $\lfloor\frac{n}{2}\rfloor$ colors by Proposition \ref{pro2.1}.

Next, we will give an edge-coloring $c$ of $G_t$ using at most $\lfloor\frac{n}{2}\rfloor$ colors. Since $G_1$ is Hamiltonian, It follows from Corollary \ref{cor2.6} that we can color the edges of $G_1$ with two different colors from $\{1,2\}$ such that for every triple of vertices in $G_1$, there exists a proper tree in $G_1$ connecting them. Then, we color the end-edges of $P_{2j-4}$ and $P_{2j-3}$ with fresh color $j$ for $3 \le j \le \lceil\frac{t+3}{2}\rceil$. Finally, we color the internal edges of $P_i \ (2 \le i \le t)$ with two colors from $\{1,2\}$ such that $P_i$ is a proper path if $\ell(P_i) \ge 3$. One can see that we color all the edges of $G_t$ with $\lceil\frac{t+3}{2}\rceil$ colors. Since $n_0+\sum_{i=1}^t(\ell(P_i)-1)=n$ and $n_0 \ge 4$, we have that $\lceil\frac{t+3}{2}\rceil \le \lfloor\frac{n}{2}\rfloor$, the equality holds if and only if $n_0=4$ and $\ell(P_i)=2$.

Now we show that this edge-coloring is a 3-proper coloring of $G_t$. We apply induction on $t$ $(t\geq2)$. If $t=2$, then let $u,v,w$ be any three vertices of $G_2$. If all of $\{u,v,w\}$ are in $G_1$, then there is already a proper tree connecting them in $G_1$. If two of $\{u,v,w\}$ are in $G_1$, without loss of generality, assume that $u \in V(P_2)\setminus \{a_2,b_2\}$, then the proper tree connecting $a_2, v, w$ in $G_1$ together with the proper path $uP_2a_2$ forms a proper $\{u,v,w\}$-tree in $G_2$. If one of $\{u,v,w\}$ is in $G_1$, without loss of generality, assume that $u,v \in V(P_2)\setminus \{a_2,b_2\}$ and $v$ is on the proper path $uP_2a_2$, then the proper tree connecting $a_2,w$ in $G_1$ together with the proper path $uP_2a_2$ forms a proper $\{u,v,w\}$-tree in $G_2$. If none of $\{u,v,w\}$ is in $G_1$, then $\{u,v,w\} \subset V(P_2)\setminus \{a_2,b_2\}$. Thus, there is already a proper path connecting them in $P_2$. Now we assume that this edge-coloring makes $G_i \ (1 \le i \le t-1)$ 3-proper connected. It is sufficient to show that this edge-coloring makes $G_t$ 3-proper connected. For any three vertices $\{u,v,w\}$ of $G_t$, if all of them are in $G_{t-1}$, then there is already a proper tree in $G_{t-1}$ connecting them. If two of $\{u,v,w\}$ are in $G_{t-1}$, without loss of generality, say $u \in V(P_t)\setminus \{a_t, b_t\}$. If $t$ is even, then the color of the end-edges of $P_t$ does not appear in $G_{t-1}$. Thus, the proper tree connecting $a_t, v, w$ in $G_{t-1}$ together with the proper path $uP_ta_t$ forms a proper $\{u,v,w\}$-tree in $G_t$. If $t$ is odd, then the end-edges of $P_{t-1}$ and $P_t$ have the same color which does not appear in $G_{t-2}$. We consider the following two cases.

\textbf{Case 1.} $|[V(P_t) \cap V(P_{t-1})]\setminus V(G_{t-2})| \le 1$.

Without loss of generality, assume that $a_t \in V(G_{t-2})$ and $a_t \ne b_{t-1}$. If both of $v$ and $w$ are in $G_{t-2}$, then the proper tree connecting $a_t, v, w$ in $G_{t-2}$ together with the proper path $uP_ta_t$ forms a proper $\{u,v,w\}$-tree in $G_t$. If $v \in V(G_{t-2})$ and $w \in V(P_{t-1})\setminus \{a_{t-1},b_{t-1}\}$, then the proper tree connecting $a_t, v, b_{t-1}$ in $G_{t-2}$ together with the proper paths $uP_ta_t$ and $wP_{t-1}b_{t-1}$ forms a proper $\{u,v,w\}$-tree in $G_t$. If $v, w \in V(P_{t-1})\setminus \{a_{t-1},b_{t-1}\}$, without loss of generality, assume that $v$ is on the proper path $wP_{t-1}b_{t-1}$. Thus, the proper tree connecting $a_t, b_{t-1}$ in $G_{t-2}$ together with the proper paths $uP_ta_t$ and $wP_{t-1}b_{t-1}$ forms a proper $\{u,v,w\}$-tree in $G_t$.

\textbf{Case 2.} $|[V(P_t) \cap V(P_{t-1})]\setminus V(G_{t-2})| =2$.

One can see that $\ell(P_{t-1}) \ge 3$. Without loss of generality, assume that $a_t$ is on the proper path of $b_tP_{t-1}a_{t-1}$ and $b_t$ is on the proper path of $a_tP_{t-1}b_{t-1}$. If both of $v$ and $w$ are in $G_{t-2}$, then the proper tree connecting $b_{t-1}, v, w$ in $G_{t-2}$ together with the proper path $uP_ta_tP_{t-1}b_{t-1}$ forms a proper $\{u,v,w\}$-tree in $G_t$. If $v \in V(G_{t-2})$ and $w \in V(P_{t-1})\setminus \{a_{t-1},b_{t-1}\}$, without loss of generality, assume that $w$ is on the proper path $a_tP_{t-1}b_{t-1}$, then the proper tree connecting $v, b_{t-1}$ in $G_{t-2}$ together with the proper path $uP_ta_tP_{t-1}b_{t-1}$ forms a proper $\{u,v,w\}$-tree in $G_t$. If $v, w \in V(P_{t-1})\setminus \{a_{t-1},b_{t-1}\}$, without loss of generality, assume that $v$ is on the proper path $a_tP_{t-1}b_{t-1}$. If $w$ is on the proper path $a_tP_{t-1}b_{t-1}$, then the path $uP_ta_tP_{t-1}b_{t-1}$ is a proper path connecting $u,v,w$ in $G_t$. If $w$ is on the proper path $a_tP_{t-1}a_{t-1}$, then the proper tree connecting $a_{t-1}, b_{t-1}$ in $G_{t-2}$ together with the proper paths $uP_ta_tP_{t-1}b_{t-1}$ and $wP_{t-1}a_{t-1}$ forms a proper $\{u,v,w\}$-tree in $G_t$.

If one of $\{u,v,w\}$ is in $G_{t-1}$, then we can easily get a proper $\{u,v,w\}$-tree in $G_{t}$ in a way similar to the situation that two of $\{u,v,w\}$ are in $G_{t-1}$. If none of $\{u,v,w\}$ is in $G_{t-1}$, then $\{u,v,w\} \subset V(P_t)\setminus \{a_t,b_t\}$. Thus, there is also already a proper path in $P_t$ connecting them. Hence, we complete the proof.
\end{proof}

\end{document}